\documentclass[11pt]{amsart}
\usepackage[margin=1.25in, headheight=14pt]{geometry}

\usepackage{amsmath,amsfonts,amssymb,amsthm,amsopn}
\usepackage{graphics}
\usepackage[usenames,dvipsnames]{xcolor}

\numberwithin{equation}{section}
\usepackage{thmtools}
\usepackage{thm-restate}
\usepackage{cleveref}
\usepackage{setspace}

\declaretheorem[numberwithin=section]{theorem}
\declaretheorem[sibling=theorem]{lemma}
\declaretheorem[sibling=theorem]{corollary}
\title{Polygonal Dissections and Reversions of Series}
\author{
        Alison Schuetz, Gwyn Whieldon
        }
\date{\today}
\usepackage{tikz}
\usetikzlibrary{calc,patterns,matrix,arrows}
\usepgflibrary{arrows}
\usepackage{enumerate}
\graphicspath{{FussCatalan/}}

\theoremstyle{definition}
\newtheorem{definition}[theorem]{Definition}

\newtheorem{example}[theorem]{Example}
\newtheorem{proposition}[theorem]{Proposition}
\newtheorem{remark}[theorem]{Remark}
\newtheorem{question}[theorem]{Question}

\newcommand{\CC}{{\mathbb C}}

\newcommand{\MM}{{\mathcal M}}

\newcommand{\matching}[2]{\displaystyle \mathop{\updownarrow}^{#1}_{#2}}

\begin{document}

\let\thefootnote\relax\footnotetext{Both authors were supported by the Summer Research Institute at Hood College. The second author also received supported from a Board of Associates Grant from Hood College.}

%%%%%%%%%%%%%%%%%%%%%%%%%%%%

%%%%%%%%%%%%%%%%%%%%%%%%%%%%

%%%%%%%%%%%%%%%%%%%%%%%%%%%%

%%%%%%%%%%%%%%%%%%%%%%%%%%%%

% Turn this on during editing later today after finished.
\begin{abstract}
The Catalan numbers $C_k$ were first studied by Euler, in the context of enumerating triangulations of polygons $P_{k+2}$. Among the many generalizations of this sequence, the Fuss-Catalan numbers $C^{(d)}_k$ count enumerations of dissections of polygons $P_{k(d-1)+2}$ into $(d+1)$-gons. In this paper, we provide a formula enumerating polygonal dissections of $(n+2)$-gons, classified by partitions $\lambda$ of $[n]$. We connect these counts $a_{\lambda}$ to reverse series arising from iterated polynomials. Generalizing this further, we show that the coefficients of the reverse series of polynomials $x=z-\sum_{j=0}^{\infty} b_j z^{j+1}$ enumerate colored polygonal dissections.
\end{abstract}

\maketitle
\section{Catalan Numbers and Polygonal Partitions}\label{sec:intro}

The Catalan numbers $C_k=\frac{1}{k+1}\binom{2k}{k}$ for $n\geq 0$ are the answer to myriad counting problems (see \cite{EC1}, \cite{stan05}, \cite{MR2179858}.) For example, they count the number of triangulations of an $(k+2)$-gon, the number of noncrossing handshake-pairings of $2k$ people seated at a round table, the number of binary rooted trees with $k$ internal nodes, the number of Dyck paths of length $2k$, and the number of noncrossing partitions of $k$.

\begin{figure}[h!]
\begin{center}
\begin{tikzpicture}[scale=0.5,auto=left,vertices/.style={circle, fill=black, inner sep=0.5pt}]

\draw (0,1.8)--++(72:1)--++(144:1)--++(216:1)--++(288:1)--cycle;
\draw[color=MidnightBlue](0,1.8)++(144:1.6)--++(0:1.6);
\draw[color=MidnightBlue](0,1.8)++(144:1.6)--++(324:1.6);

\draw (2.1,1.8)--++(72:1)--++(144:1)--++(216:1)--++(288:1)--cycle;
\draw[color=MidnightBlue] (2.1,1.8)++(108:1.6)--++(252:1.6);
\draw[color=MidnightBlue] (2.1,1.8)++(108:1.6)--++(288:1.6);

\draw (4.2,1.8)--++(72:1)--++(144:1)--++(216:1)--++(288:1)--cycle;
\draw[color=MidnightBlue] (4.2,1.8)++(72:1)--++(180:1.6);
\draw[color=MidnightBlue] (4.2,1.8)++(72:1)--++(216:1.6);

\draw (1,0)--++(72:1)--++(144:1)--++(216:1)--++(288:1)--cycle;
\draw[color=MidnightBlue] (1,0)--++(144:1.6);
\draw[color=MidnightBlue] (1,0)--++(108:1.6);

\draw (3.2,0)--++(72:1)--++(144:1)--++(216:1)--++(288:1)--cycle;
\draw[color=MidnightBlue] (3.2,0)++(0:-1)--++(72:1.6);
\draw[color=MidnightBlue] (3.2,0)++(0:-1)--++(36:1.6);

\end{tikzpicture}\;\;\;\;\;\;\;\;\;\;\begin{tikzpicture}[scale=0.5,auto=left,vertices/.style={circle, fill=black, inner sep=0.75pt}]

\foreach \a/\b in {0/1.8,2.1/1.8,4.2/1.8,1/0,3.2/0} {
\draw[color=black!40] (\a,\b) circle (0.7);
\node[vertices] at (\a+0.7,\b) {};
\node[vertices] at (\a+0.5*0.7,\b+0.7*0.866) {};
\node[vertices] at (\a-0.5*0.7,\b+0.7*0.866) {};
\node[vertices] at (\a-0.7,\b) {};
\node[vertices] at (\a-0.5*0.7,\b-0.7*0.866) {};
\node[vertices] at (\a+0.5*0.7,\b-0.7*0.866) {};
}

\foreach \a/\b in {0/1.8} {
\path[color=BlueViolet] (\a+0.7,\b) edge [bend left=60] (\a+0.5*0.7,\b+0.7*0.866);
\path[color=BlueViolet] (\a-0.5*0.7,\b+0.7*0.866) edge [bend left=60] (\a-0.7,\b);
\path[color=BlueViolet] (\a-0.5*0.7,\b-0.7*0.866) edge [bend left=60] (\a+0.5*0.7,\b-0.7*0.866);
}

\foreach \a/\b in {2.1/1.8} {
\path[color=BlueViolet] (\a-0.5*0.7,\b+0.7*0.866) edge [bend right=60] (\a+0.5*0.7,\b+0.7*0.866);
\path[color=BlueViolet] (\a+0.7,\b) edge (\a-0.7,\b);
\path[color=BlueViolet] (\a-0.5*0.7,\b-0.7*0.866) edge [bend left=60] (\a+0.5*0.7,\b-0.7*0.866);
}

\foreach \a/\b in {4.2/1.8} {
\path[color=BlueViolet] (\a+0.5*0.7,\b+0.7*0.866) edge [bend left=60] (\a-0.5*0.7,\b+0.7*0.866);
\path[color=BlueViolet] (\a-0.7,\b) edge [bend left=60] (\a-0.5*0.7,\b-0.7*0.866) ;
\path[color=BlueViolet] (\a+0.5*0.7,\b-0.7*0.866) edge [bend left=60] (\a+0.7,\b);
}

\foreach \a/\b in {1/0} {
\path[color=BlueViolet] (\a+0.7,\b) edge [bend right=60] (\a+0.5*0.7,\b-0.7*0.866);
\path[color=BlueViolet] (\a-0.5*0.7,\b+0.7*0.866) edge [bend left=60] (\a-0.7,\b);
\path[color=BlueViolet] (\a-0.5*0.7,\b-0.7*0.866) edge (\a+0.5*0.7,\b+0.7*0.866);
}

\foreach \a/\b in {3.2/0} {
\path[color=BlueViolet] (\a+0.7,\b) edge [bend left=60] (\a+0.5*0.7,\b+0.7*0.866);
\path[color=BlueViolet] (\a-0.5*0.7,\b-0.7*0.866) edge [bend right=60] (\a-0.7,\b);
\path[color=BlueViolet] (\a-0.5*0.7,\b+0.7*0.866) edge (\a+0.5*0.7,\b-0.7*0.866);
}
\end{tikzpicture}\;\;\;\;\;\;\;\;\;\;\begin{tikzpicture}[xscale=0.5, yscale=-0.5,auto=left,vertices/.style={circle, fill=black, inner sep=0.6pt}]

%%%%%%%%%%%%%%%%%%%%%%%%%%%%%%%%%%%%
%% Tree 1
%%%%%%%%%%%%%%%%%%%%%%%%%%%%%%%%%%%%
\foreach \a/\b in {0/1.8,-0.4/2.2,0/2.6}{
\draw[color=Plum,-] (\a,\b)--(\a-0.4,\b+0.4) (\a,\b)--(\a+0.4,\b+0.4);
}

\node[vertices] at (0,1.8) {};
\foreach \a/\b in {0/1.8,-0.4/2.2,0/2.6}{
\node[vertices] at (\a-0.4,\b+0.4){};
\node[vertices] at (\a+0.4,\b+0.4){};
}
%%%%%%%%%%%%%%%%%%%%%%%%%%%%%%%%%%%%

%%%%%%%%%%%%%%%%%%%%%%%%%%%%%%%%%%%%
%% Tree 2
%% Because of spacing issues, this one looks slightly different than the others
%%%%%%%%%%%%%%%%%%%%%%%%%%%%%%%%%%%%
\foreach \a/\b in {2.1/1.8}{
\draw[color=Plum,-] (\a,\b)--(\a-0.4,\b+0.4) (\a,\b)--(\a+0.4,\b+0.4);
}
\foreach \a/\b in {1.7/2.2}{
\draw[color=Plum,-] (\a,\b)--(\a-0.4,\b+0.4) (\a,\b)--(\a+0.3,\b+0.4);
}
\foreach \a/\b in {2.5/2.2}{
\draw[color=Plum,-] (\a,\b)--(\a-0.3,\b+0.4) (\a,\b)--(\a+0.4,\b+0.4);
}
\node[vertices] at (2.1,1.8) {};
\foreach \a/\b in {2.1/1.8}{
\node[vertices] at (\a-0.4,\b+0.4){};
\node[vertices] at (\a+0.4,\b+0.4){};
}
\node[vertices] at (1.3,2.6){};
\node[vertices] at (2,2.6){};
\node[vertices] at (2.9,2.6){};
\node[vertices] at (2.2,2.6){};
%%%%%%%%%%%%%%%%%%%%%%%%%%%%%%%%%%%%

%%%%%%%%%%%%%%%%%%%%%%%%%%%%%%%%%%%%
%% Tree 3
%%%%%%%%%%%%%%%%%%%%%%%%%%%%%%%%%%%%
\foreach \a/\b in {4.2/1.8,4.6/2.2,4.2/2.6}{
\draw[color=Plum,-] (\a,\b)--(\a-0.4,\b+0.4) (\a,\b)--(\a+0.4,\b+0.4);
}

\node[vertices] at (4.2,1.8) {};
\foreach \a/\b in {4.2/1.8,4.6/2.2,4.2/2.6}{
\node[vertices] at (\a-0.4,\b+0.4){};
\node[vertices] at (\a+0.4,\b+0.4){};
}
%%%%%%%%%%%%%%%%%%%%%%%%%%%%%%%%%%%%

%%%%%%%%%%%%%%%%%%%%%%%%%%%%%%%%%%%%
%% Tree 4
%%%%%%%%%%%%%%%%%%%%%%%%%%%%%%%%%%%%
\foreach \a/\b in {1/0,0.6/0.4,0.2/0.8}{
\draw[color=Plum,-] (\a,\b)--(\a-0.4,\b+0.4) (\a,\b)--(\a+0.4,\b+0.4);
}

\node[vertices] at (1,0) {};
\foreach \a/\b in {1/0,0.6/0.4,0.2/0.8}{
\node[vertices] at (\a-0.4,\b+0.4){};
\node[vertices] at (\a+0.4,\b+0.4){};
}
%%%%%%%%%%%%%%%%%%%%%%%%%%%%%%%%%%%%

%%%%%%%%%%%%%%%%%%%%%%%%%%%%%%%%%%%%
%% Tree 5
%%%%%%%%%%%%%%%%%%%%%%%%%%%%%%%%%%%%
\foreach \a/\b in {3.2/0,3.6/0.4,4/0.8}{
\draw[color=Plum,-] (\a,\b)--(\a-0.4,\b+0.4) (\a,\b)--(\a+0.4,\b+0.4);
}

\node[vertices] at (3.2,0) {};
\foreach \a/\b in {3.2/0,3.6/0.4,4/0.8}{
\node[vertices] at (\a-0.4,\b+0.4){};
\node[vertices] at (\a+0.4,\b+0.4){};
}
%%%%%%%%%%%%%%%%%%%%%%%%%%%%%%%%%%%%

\end{tikzpicture}\\
{\tiny Triangulations of $(k+2)$-gon}\;\;\;{\tiny Noncrossing Pairings of $2k$ People}\;\;\;{\tiny Binary Rooted Trees of $k$-pairs} 
\end{center}
\vspace{10pt}
\begin{center}
\begin{tikzpicture}[scale=0.5,auto=left,vertices/.style={circle, fill=black, inner sep=0.5pt}]

\foreach \a/\b in {0/1.8,2.1/1.8,4.2/1.8,1.0/0,3.2/0}{
\draw[color=black!40,-] (\a,\b)--(\a+1.5,\b);
\foreach \n in {0,1,2,3,4,5,6}{
\draw[-] (\a+0.25*\n,\b+0.05)--(\a+0.25*\n,\b-0.05);
}};

\draw[-,color=Brown] (0,1.8)--++(0.25,0.25)--++(0.25,0.25)--++(0.25,-0.25)--++(0.25,-0.25)--++(0.25,0.25)--++(0.25,-0.25);

\draw[-,color=Brown] (2.1,1.8)--++(0.25,0.25)--++(0.25,-0.25)--++(0.25,0.25)--++(0.25,-0.25)--++(0.25,0.25)--++(0.25,-0.25);
\draw[-,color=Brown] (4.2,1.8)--++(0.25,0.25)--++(0.25,-0.25)--++(0.25,0.25)--++(0.25,0.25)--++(0.25,-0.25)--++(0.25,-0.25);
\draw[-,color=Brown] (1,0)--++(0.25,0.25)--++(0.25,0.25)--++(0.25,0.25)--++(0.25,-0.25)--++(0.25,-0.25)--++(0.25,-0.25);
\draw[-,color=Brown] (3.2,0)--++(0.25,0.25)--++(0.25,0.25)--++(0.25,-0.25)--++(0.25,0.25)--++(0.25,-0.25)--++(0.25,-0.25);

\end{tikzpicture}\;\;\;\;\;\;\;\;\;\;\begin{tikzpicture}[scale=0.5,auto=left,vertices/.style={circle, fill=black, inner sep=0.5pt}]

\filldraw[fill=Aquamarine!60] (0,1.8)--++(0,1)--++(0.3,0)--++(0,-1);
\filldraw[fill=Aquamarine!60] (0.6,1.8)--++(0,1)--++(0.9,0)--++(0,-1)--++(-0.3,0)--++(0,0.7)--++(-0.3,0)--++(0,-0.7);

\filldraw[fill=Aquamarine!60] (2.7,1.8)--++(0,0.4)--++(0.3,0)--++(0,-0.4);
\filldraw[fill=Aquamarine!60] (2.1,1.8)--++(0,1)--++(1.5,0)--++(0,-1)--++(-0.3,0)--++(0,0.7)--++(-0.9,0)--++(0,-0.7);

\filldraw[fill=Aquamarine!60] (5.4,1.8)--++(0,1)--++(0.3,0)--++(0,-1);
\filldraw[fill=Aquamarine!60] (4.2,1.8)--++(0,1)--++(0.9,0)--++(0,-1)--++(-0.3,0)--++(0,0.7)--++(-0.3,0)--++(0,-0.7);

\filldraw[fill=Aquamarine!60] (1.0,0)--++(0,1)--++(0.3,0)--++(0,-1);
\filldraw[fill=Aquamarine!60] (1.6,0)--++(0,1)--++(0.3,0)--++(0,-1);
\filldraw[fill=Aquamarine!60] (2.2,0)--++(0,1)--++(0.3,0)--++(0,-1);

\filldraw[fill=Aquamarine!60] (3.2,0)--++(0,1)--++(1.5,0)--++(0,-1)--++(-0.3,0)--++(0,0.7)--++(-0.3,0)--++(0,-0.7)--++(-0.3,0)--++(0,0.7)--++(-0.3,0)--++(0,-0.7);

\foreach \a/\b in {0/1.8,2.1/1.8,4.2/1.8,1.0/0,3.2/0}{
\foreach \n in {0,2,4}{
\draw[-] (\a+0.3*\n,\b)--(\a+0.3*\n+0.3,\b);
}};

\end{tikzpicture}\\
{\tiny Dyck Paths of Length $2k$}\;\;\;\;\;\;{\tiny Noncrossing Partitions of $[k]$}
\end{center}
\caption{Examples of Sets Counted by Catalan Number $C_3=5$}
\end{figure}
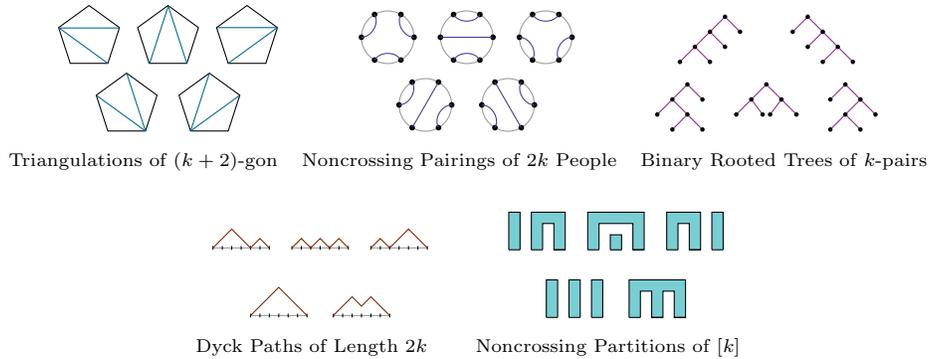

In this paper, we will alternate between the recursive definition of the Catalan numbers and the closed formula for $C_k$.

\begin{theorem}[Catalan Recursion~\cite{EC1}]\label{thm:catrec} Consider sequence $\{C_0,C_1,C_2,...\}$. Suppose $C_0=1$, and $C_{k+1}=\sum_{i=0}^{k}C_i\cdot C_{k-i}$. Then $C_k$ are the Catalan numbers, $C_k=\frac{1}{k+1}\binom{2k}{k}$.
\end{theorem}

There is a similar recursive formula for the Fuss-Catalan numbers $C_k^{(d)}=\frac{1}{k(d-1)+1}\binom{dk}{k}$, which specializes to the Catalan numbers when $d=2$.

\begin{theorem}[Fuss-Catalan Recursion~\cite{MR0292690}]\label{thm:gencatrec} Consider sequence $\{C_0^{(d)},C_1^{(d)},C_2^{(d)},...\}$. Suppose $C_0^{(d)}=1$, and $C_{k+1}^{(d)}=\sum_{k_1+k_2+\cdots+k_d=n}C_{k_1}^{(d)}C_{k_2}^{(d)}\cdots C_{k_d}^{(d)}$. Then $C_k^{(d)}$ are the generalized Catalan (or Fuss-Catalan) numbers $C_k^{(d)}=\frac{1}{k(d-1)+1}\binom{dk}{k}$.
\end{theorem}

There is a well-known bijection between triangulations of $(k+2)$-gons $P_{k+2}$ and binary rooted trees with $k$ internal nodes (see \cite{MR1783940} for a history of this problem.) There is also a bijection between partitions of a $(k(d-1)+2)$-gon $P_{k(d-1)+2}$ into $(d+1)$-gons and $d$-ary trees with $k$ internal nodes (see Hilton-Pedersen~\cite{MR1098222}.) The wording of Theorem 0.2 from \cite{MR1098222} has been changed slightly to reflect the notation used in this note.

\begin{theorem}[Theorem 0.2, \cite{MR1098222}]\label{thm:bijection1}Let $P_k^d$ denote the number of ways of subdividing a convex polygon into $k$ disjoint $(d+1)$-gons by means of nonintersecting diagonals, $k\geq 1$, and let $A_k^d$ denote the number of $d$-ary trees with $k$-internal nodes, $k\geq 1$. Then $P_k^d=A_k^d=C_k^{(d)}$ for all $d\geq 2$, $k\geq 1$.
\end{theorem}

\begin{proof}That the number of $d$-ary rooted trees with $k$ internal nodes can be counted by the Fuss-Catalan numbers can easily be shown inductively via the generalized Catalan recursion formula. For $k=1$, there are precisely $A_k^d=1$ such trees. For $k\geq 1$, each tree in the set $A_{k+1}^d$ consists of an internal node with $d$ branches and some rooted $d$-ary tree (possibly empty) attached to each branch (see Figure~\ref{fig:branches}.)
\begin{figure}[h!]
\begin{center}
\begin{tikzpicture}[scale=0.5,auto=left,vertices/.style={circle, fill=black, inner sep=1pt}]
%%Root node:
\node[vertices] at (0,0){};

%% Branches and A_k:
\foreach \xval in {-4,-2.25,-0.5,2.25,4}{
\draw[Plum,-] (0,0)--(\xval,-2);
\node[vertices] at (\xval,-2){};
\filldraw[-,black,fill=Plum!15] (\xval,-2)--++(0.6,-1)--++(0,-0.75)--++(-1.2,0)--++(0,0.75)--++(0.6,1);
}

%%Labels for sub-trees:
\node at (-4,-3.2) {\tiny $A_{k_1}^d$};
\node at (-2.25,-3.2) {\tiny $A_{k_2}^d$};
\node at (-0.5,-3.2) {\tiny $A_{k_3}^d$};
\node at (1,-3.2) {\footnotesize $\cdots$};
\node at (2.25,-3.2) {\footnotesize $\cdot$};
\node at (4,-3.2) {\tiny $A_{k_d}^d$};

\end{tikzpicture}
\end{center}
\caption{Recursive Construction of $A_{k+1}^d$}
\label{fig:branches}
\end{figure}
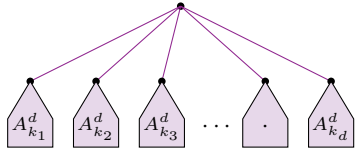

As there are $k$ remaining internal nodes to partition amongst the branches, we have $k_1+\cdots+k_d=k$, and our $d$-ary rooted trees with $(k+1)$ internal nodes must satisfy the recursion formula:
$$A_{k+1}^d=\sum_{k_1+k_2+\cdots +k_d=k}A_{k_1}^dA_{k_2}^d\cdots A_{k_d}^d$$
By Theorem~\ref{thm:gencatrec}, we have that the number of $d$-ary rooted trees $A_k^d=C_{k}^{(d)}$.

There is a bijection (illustrated in Figure~\ref{fig:bijection1}) between $d$-ary rooted trees with $k$ internal nodes and subdivisions of convex polygons into $k$ disjoint $(d+1)$-gons via diagonals. Choose one edge to correspond to the root, then draw the $d$ branches from that root to the $d$ other edges of the $(d+1)$-gon in the dissection. Continue this, treating each vertex lying on a diagonal as the new root of a sub-tree.

\begin{figure}[h!]
\begin{center}
\begin{tikzpicture}[scale=0.25,auto=left,vertices1/.style={circle, fill=black, inner sep=0.1pt},roots/.style={circle, fill=Plum, inner sep=1pt},internalnode/.style={circle, fill=Plum, inner sep=1.5pt}]

\foreach \angle/\name in {0/a, 36/b, 72/c, 108/d, 144/e, 180/f, 216/g, 252/h, 288/i, 324/j} {
\node[vertices1] (\name) at (\angle:4) {};
};

\foreach \alpha/\beta in {a/b,b/c,c/d,d/e,e/f,f/g,g/h,h/i,i/j,j/a}{
\draw[-] (\alpha)--(\beta);
};

\foreach \alpha/\beta in {b/e,e/j,f/i}{
\draw[MidnightBlue,-] (\alpha)--(\beta);
};

\node at (0,-10){};

\end{tikzpicture}\;\;\;\;\;\;\begin{tikzpicture}[scale=0.5,auto=left,vertices1/.style={circle, fill=black, inner sep=0.1pt},roots/.style={circle, fill=Plum, inner sep=1pt},internalnode/.style={circle, fill=Plum, inner sep=1.5pt}]

\foreach \angle/\name in {0/a, 36/b, 72/c, 108/d, 144/e, 180/f, 216/g, 252/h, 288/i, 324/j} {
\node[vertices1] (\name) at (\angle:4) {};
};

\foreach \alpha/\beta in {a/b,b/c,c/d,d/e,e/f,f/g,g/h,h/i,i/j,j/a}{
\draw[-] (\alpha)--(\beta);
};

\foreach \alpha/\beta in {b/e,e/j,f/i}{
\draw[MidnightBlue,-,thick] (\alpha)--(\beta);
};

\foreach \angle/\name in {0/aa, 36/bb, 108/dd, 144/ee, 180/ff, 216/gg, 252/hh, 288/ii, 324/jj} {
\node[roots] (\name) at (\angle+18:3.8) {};
};

\node[internalnode] (cc) at (72+18:3.8) {};
\node[internalnode] (kk) at (72+18:2.34) {};
\node[internalnode] (ll) at (0:0) {};
\node[internalnode] (mm) at (216+18:2.34) {};

\foreach \alpha/\beta in {cc/dd,cc/bb,cc/kk,kk/ll,kk/aa,kk/jj,ll/mm,ll/ee,ll/ii,mm/ff,mm/gg,mm/hh}{
\draw[Plum,-] (\alpha)--(\beta);
}

\end{tikzpicture}\;\;\;\;\;\;\;\;\begin{tikzpicture}[scale=0.2,auto=left,vertices1/.style={circle, fill=black, inner sep=0.1pt},roots/.style={circle, fill=Plum, inner sep=1pt},internalnode/.style={circle, fill=Plum, inner sep=1.3pt}]

\node[internalnode] (cc) at (0,0) {};
\node[roots] (bb) at (3,-3){};
\node[roots] (dd) at (-3,-3){};

\node[internalnode] (kk) at (0,-3) {};
\node[roots] (aa) at (3,-6) {};
\node[roots] (jj) at (0,-6) {};

\node[internalnode] (ll) at (-3,-6) {};
\node[roots] (ii) at (0,-9) {};
\node[roots] (ee) at (-6,-9) {};

\node[internalnode] (mm) at (-3,-9) {};
\node[roots] (ff) at (0,-12) {};
\node[roots] (gg) at (-3,-12) {};
\node[roots] (hh) at (-6,-12) {};

\foreach \alpha/\beta in {cc/dd,cc/bb,cc/kk,kk/ll,kk/aa,kk/jj,ll/mm,ll/ee,ll/ii,mm/ff,mm/gg,mm/hh}{
\draw[Plum,-] (\alpha)--(\beta);
}

\node at (0,-18){};
\end{tikzpicture}
\end{center}
\caption{Bijection between $(d+1)$ dissections and rooted $d$-ary trees.}
\label{fig:bijection1}
\end{figure}
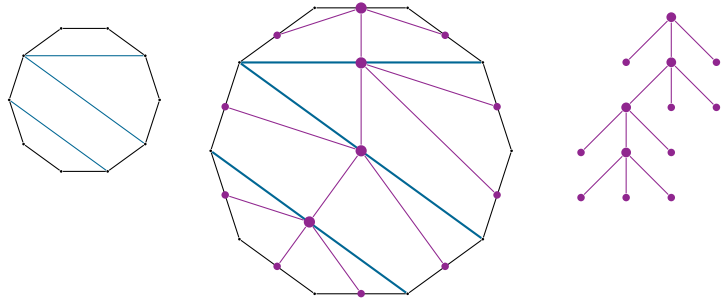

This process is easy to reverse (constructing a unique polygonal dissection from a rooted tree), i.e. given a polygonal dissection into $(d+1)$-gons, we have a unique $d$-ary rooted tree, and vice versa. This provides our desired bijection, and $P_k^d=A_k^d=C_k^{(d)}$.
\end{proof}

There have been several papers enumerating general polygonal dissections of $n$-gons with $k$ nonintersecting diagonals (see \cite{motzkin}, \cite{ronread} and \cite{MR2252931}.) We consider here a case that does not currently appear in the literature: enumerating polygonal dissections of convex $n$-gons $P_n$, where each piece of the dissection must be a $(d+1)$-gon, where $d\in\{d_1,d_2,...,d_k\}$ for fixed, distinct integers $d_i\geq 2$. To standardize terminology and indices for \emph{polygonal dissections}, we include precise definitions here.

\begin{definition}[Polygonal Dissections] A \emph{polygonal dissection} of a convex $n$-gon is the union of the polygon and any nonintersecting subset of its diagonals. A \emph{$d$-dissection} (respectively, a \emph{$(d_1,d_2,...,d_r)$-dissection}) is a polygonal dissection such that the regions formed by the dissection are all convex $(d+1)$-gons (respectively, each region is a $(d+1)$-gon, where $d\in\{d_1,d_2,...,d_r\}$).
\end{definition}

\begin{definition}[Type of a Polygonal Dissection] Let $\lambda$ be a partition of $n$ with $k_j$ parts of size $j$. We say a dissection of an $(n+2)$-gon consisting of $k_j$ $(j+2)$-gons is a \emph{polygonal dissection of type $\lambda$}, and denote the set of all such polygonal dissections as $P_{\lambda,n}$.
\end{definition}
\begin{figure}[h!]
\begin{center}
\begin{tikzpicture}[scale=0.35,auto=left,vertices1/.style={circle, fill=black, inner sep=0.1pt},roots/.style={circle, fill=Plum, inner sep=1pt},internalnode/.style={circle, fill=Plum, inner sep=1.5pt}]

\foreach \angle/\name in {0/a, 36/b, 72/c, 108/d, 144/e, 180/f, 216/g, 252/h, 288/i, 324/j} {
\node[vertices1] (\name) at (\angle:4) {};
};

\foreach \alpha/\beta in {a/b,b/c,c/d,d/e,e/f,f/g,g/h,h/i,i/j,j/a}{
\draw[-] (\alpha)--(\beta);
};

\foreach \alpha/\beta in {b/e,e/j,f/i}{
\draw[MidnightBlue,-] (\alpha)--(\beta);
};

\end{tikzpicture}\;\;\;\;\;\;\;\;\;\;\;\;\;\;\begin{tikzpicture}[scale=0.35,auto=left,vertices1/.style={circle, fill=black, inner sep=0.1pt},roots/.style={circle, fill=Plum, inner sep=1pt},internalnode/.style={circle, fill=Plum, inner sep=1.5pt}]

\foreach \angle/\name in {0/a, 36/b, 72/c, 108/d, 144/e, 180/f, 216/g, 252/h, 288/i, 324/j} {
\node[vertices1] (\name) at (\angle:4) {};
};

\foreach \alpha/\beta in {a/b,b/c,c/d,d/e,e/f,f/g,g/h,h/i,i/j,j/a}{
\draw[-] (\alpha)--(\beta);
};

\foreach \alpha/\beta in {i/b,f/i,b/f}{
\draw[MidnightBlue,-] (\alpha)--(\beta);
};

\end{tikzpicture}
\end{center}
\caption{A $3$-dissection (left) of type $\lambda=2+2+2+2$ and a $(2,3,4)$-dissection (right) of type $\lambda=3+2+2+1$ of a $10$-gon.}
\end{figure}
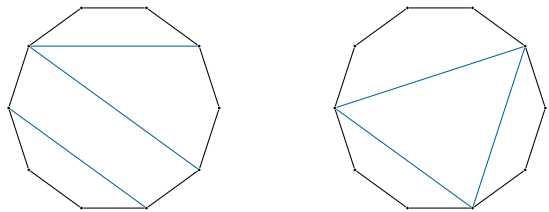

Polygonal dissections of an $(n(d-1)+1)$-gon into $(d+1)$-gons are in bijection with $d$-ary rooted trees with $n+1$ internal nodes, as shown in Theorem~\ref{thm:bijection1}. Similarly, polygonal dissections of type $\lambda$ are in bijection with rooted plane trees with a particular downdegree sequence.

\begin{definition}[Rooted Trees and Downdegree Sequences] A \emph{rooted plane tree} is a tree $T$ with a distinguished vertex called the \emph{root}. The \emph{downdegree sequence} ${\bf r}=(r_0,r_1,r_2,....,r_n)$ of a rooted tree counts the number of vertices $r_j$ with $j$ neighbors further away from the root than the vertex itself.
\end{definition}

\begin{figure}[h!]
\begin{center}
\begin{tikzpicture}[scale=0.2,auto=left,vertices1/.style={circle, fill=black, inner sep=0.1pt},roots/.style={circle, fill=Plum, inner sep=1pt},internalnode/.style={circle, fill=Plum, inner sep=1.3pt}]

\node[internalnode] (aa) at (0,0) {};

\node[internalnode] (bb) at (-6,-3){};
\node[internalnode] (cc) at (-2,-3){};
\node[roots] (dd) at (2,-3){};
\node[internalnode] (ee) at (6,-3){};

\node[roots] (ff) at (-7,-6){};
\node[internalnode] (gg) at (-6,-6){};
\node[roots] (hh) at (-5,-6){};
\node[roots] (ii) at (-3,-6){};
\node[internalnode] (jj) at (-1,-6){};
\node[roots] (kk) at (4.5,-6){};
\node[roots] (ll) at (5.5,-6){};
\node[roots] (mm) at (6.5,-6){};
\node[roots] (nn) at (7.5,-6){};

\node[roots] (oo) at (-5.25,-9){};
\node[roots] (pp) at (-6.75,-9){};

\node[roots] (qq) at (-3,-9){};
\node[roots] (rr) at (-2,-9){};
\node[roots] (ss) at (-1,-9){};
\node[roots] (tt) at (0,-9){};
\node[internalnode] (uu) at (1,-9){};

\node[roots] (vv) at (0,-12){};
\node[roots] (ww) at (1,-12){};
\node[roots] (xx) at (2,-12){};

\node at (0,-15){};

\foreach \alpha/\beta in {aa/bb,aa/cc,aa/dd,aa/ee,bb/ff,bb/gg,bb/hh,cc/ii,cc/jj,ee/kk,ee/ll,ee/mm,ee/nn,gg/oo,gg/pp,jj/qq,jj/rr,jj/ss,jj/tt,jj/uu,uu/vv,uu/ww,uu/xx}{
\draw[Plum,-] (\alpha)--(\beta);
}
\end{tikzpicture}\;\;\;\;\;\;\;\;\;\;\;\;\;\;\;\;\;\;\;\;\;\;\begin{tikzpicture}[scale=0.2,auto=left,vertices1/.style={circle, fill=black, inner sep=0.1pt},roots/.style={circle, fill=Plum, inner sep=1pt},internalnode/.style={circle, fill=Plum, inner sep=1.3pt}]

\node[internalnode] (aa) at (0,0) {};
\node[internalnode] (bb) at (-6,-3){};
\node[roots] (cc) at (0,-3){};
\node[internalnode] (dd) at (6,-3){};

\node[internalnode] (ee) at (-8.5,-6){};
\node[internalnode] (ff) at (-3.5,-6){};
\node[roots] (gg) at (4.5,-6){};
\node[roots] (hh) at (5.5,-6){};
\node[roots] (ii) at (6.5,-6){};
\node[roots] (jj) at (7.5,-6){};

\node[roots] (kk) at (-10.5,-9){};
\node[roots] (ll) at (-9.5,-9){};
\node[roots] (mm) at (-8.5,-9){};
\node[roots] (nn) at (-7.5,-9){};
\node[roots] (oo) at (-6.5,-9){};

\node[roots] (pp) at (-4.5,-9){};
\node[internalnode] (qq) at (-2.5,-9){};

\node[roots] (rr) at (-4,-12){};
\node[roots] (ss) at (-3,-12){};
\node[roots] (tt) at (-2,-12){};
\node[internalnode] (uu) at (-1,-12){};

\node[roots] (vv) at (-2,-15){};
\node[roots] (ww) at (-1,-15){};
\node[roots] (xx) at (0,-15){};

\foreach \alpha/\beta in {aa/bb,aa/cc,aa/dd,bb/ee,bb/ff,dd/gg,dd/hh,dd/ii,dd/jj,ee/kk,ee/ll,ee/mm,ee/nn,ee/oo,ff/pp,ff/qq,qq/rr,qq/ss,qq/tt,qq/uu,uu/vv,uu/ww,uu/xx}{
\draw[Plum,-] (\alpha)--(\beta);
}
\end{tikzpicture}
\end{center}
\caption{Rooted planar trees with downdegree sequence $(17,0,2,2,2,1,0,0,\ldots,0)$}
\label{fig:downdegree1}
\end{figure}
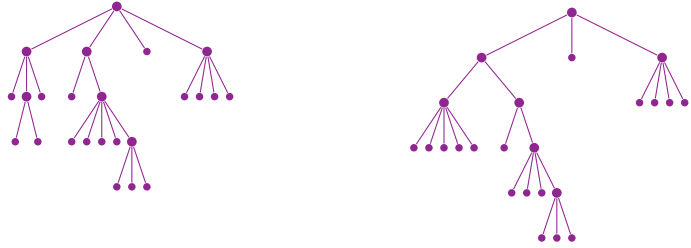

\begin{theorem}\label{thm:dissections2} Let $P_{\lambda,n}$ be the number of all polygonal dissections of type $\lambda$, where $\lambda$ is a partition of $n$ with $k_j$ parts of size $j$ and $n$ total parts. Let $T_{{\bf r},m}$ be the number of rooted plane trees on $m+1$ vertices with downdegree sequence ${\bf r}=(r_0,r_1,r_2,...,r_m)$. Then $P_{\lambda,n}=T_{{\bf r},m}$ for ${\bf r}=(n+1,0,k_1,k_2,...,k_n)$ and $m=n+k$.
\end{theorem}

\begin{proof} Let $\lambda$ be a partition of $n$ as above, and fix a polygonal dissection of type $\lambda\in P_{\lambda,n}$. We construct our tree in $T_{{\bf r},n+k}$ recursively as follows:

Choose an edge of the $(n+2)$-gon and place a root $v$ there. This will be an edge of \emph{some} $(j+2)$-gon in the dissection with $k_j\geq 1$ in $\lambda$. Place a vertex on each of the $j+1$ other sides of the $(j+2)$-gon, and connect each of these vertices to the original root vertex. The root note will have $(j+1)$ neighbors further from the root, contributing one to the value of $r_{j+1}$ in the downdegree sequence of the rooted tree.

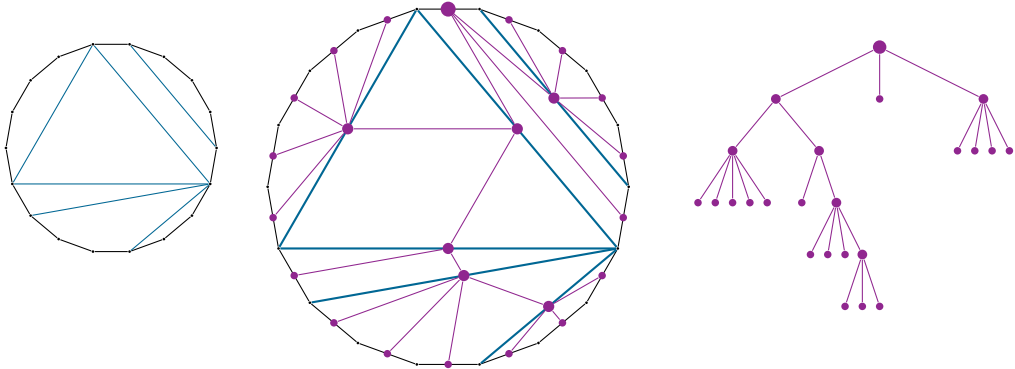
\begin{figure}[h!]
\begin{center}
\begin{tikzpicture}[scale=0.35,auto=left,vertices1/.style={circle, fill=black, inner sep=0.1pt},roots/.style={circle, fill=Plum, inner sep=1pt},internalnode/.style={circle, fill=Plum, inner sep=1.5pt}]

\foreach \angle/\name in {0/a,20/b,40/c,60/d,80/e,100/f,120/g,140/h,160/i,180/j,200/k,220/l,240/m,260/n,280/o,300/p,320/q,340/r} {
\node[vertices1] (\name) at (\angle:4) {};
};

\foreach \alpha/\beta in {a/b,b/c,c/d,d/e,e/f,f/g,g/h,h/i,i/j,j/k,k/l,l/m,m/n,n/o,o/p,p/q,q/r,r/a}{
\draw[-] (\alpha)--(\beta);
};

\foreach \alpha/\beta in {a/e,r/f,f/k,k/r,r/l,r/o}{
\draw[MidnightBlue,-] (\alpha)--(\beta);
};

\node at (0,-8){};

\end{tikzpicture}\;\;\;\;\;\;\begin{tikzpicture}[scale=0.6,auto=left,vertices1/.style={circle, fill=black, inner sep=0.1pt},roots/.style={circle, fill=Plum, inner sep=1pt},internalnode/.style={circle, fill=Plum, inner sep=1.5pt},rootnode/.style={circle, fill=Plum, inner sep=2pt}]

\node[circle, color=Plum, fill=white, inner sep=4] at (90:3.939){};

\foreach \angle/\name in {0/a,20/b,40/c,60/d,80/e,100/f,120/g,140/h,160/i,180/j,200/k,220/l,240/m,260/n,280/o,300/p,320/q,340/r} {
\node[vertices1] (\name) at (\angle:4) {};
};

\foreach \alpha/\beta in {a/b,b/c,c/d,d/e,e/f,f/g,g/h,h/i,i/j,j/k,k/l,l/m,m/n,n/o,o/p,p/q,q/r,r/a}{
\draw[-] (\alpha)--(\beta);
};

\foreach \alpha/\beta in {a/e,r/f,f/k,k/r,r/l,r/o}{
\draw[MidnightBlue,-,thick] (\alpha)--(\beta);
};

\foreach \angle/\name/\distance in {90/aa/3.939}{
\node[rootnode] (\name) at (\angle:\distance){};
}

\foreach \angle/\name/\distance in {90/aa/3.939,40/dd/3.06,40/bb/1.999,150/ee/2.57,270/ff/1.368,280/qq/1.999,310/uu/3.464}{
\node[internalnode] (\name) at (\angle:\distance){};
}

\foreach \angle/\name/\distance in {10/jj/3.939,30/ii/3.939,50/hh/3.939,70/gg/3.939,110/kk/3.939,130/ll/3.939,150/mm/3.939,170/nn/3.939,190/oo/3.939,210/pp/3.939,230/rr/3.939,250/ss/3.939,270/tt/3.939,290/vv/3.939,310/ww/3.939,330/xx/3.939,350/cc/3.939}{
\node[roots] (\name) at (\angle:\distance){};
}

\foreach \alpha/\beta in {aa/bb,aa/cc,aa/dd,bb/ee,bb/ff,dd/gg,dd/hh,dd/ii,dd/jj,ee/kk,ee/ll,ee/mm,ee/nn,ee/oo,ff/pp,ff/qq,qq/rr,qq/ss,qq/tt,qq/uu,uu/vv,uu/ww,uu/xx}{
\draw[Plum,-] (\alpha)--(\beta);
}

\end{tikzpicture}\;\;\;\;\;\;\;\;\begin{tikzpicture}[scale=0.23,auto=left,vertices1/.style={circle, fill=black, inner sep=0.1pt},roots/.style={circle, fill=Plum, inner sep=1pt},internalnode/.style={circle, fill=Plum, inner sep=1.3pt},rootnode/.style={circle, fill=Plum, inner sep=1.8pt}]

\node[rootnode] (aa) at (0,0) {};
\node[internalnode] (bb) at (-6,-3){};
\node[roots] (cc) at (0,-3){};
\node[internalnode] (dd) at (6,-3){};

\node[internalnode] (ee) at (-8.5,-6){};
\node[internalnode] (ff) at (-3.5,-6){};
\node[roots] (gg) at (4.5,-6){};
\node[roots] (hh) at (5.5,-6){};
\node[roots] (ii) at (6.5,-6){};
\node[roots] (jj) at (7.5,-6){};

\node[roots] (kk) at (-10.5,-9){};
\node[roots] (ll) at (-9.5,-9){};
\node[roots] (mm) at (-8.5,-9){};
\node[roots] (nn) at (-7.5,-9){};
\node[roots] (oo) at (-6.5,-9){};

\node[roots] (pp) at (-4.5,-9){};
\node[internalnode] (qq) at (-2.5,-9){};

\node[roots] (rr) at (-4,-12){};
\node[roots] (ss) at (-3,-12){};
\node[roots] (tt) at (-2,-12){};
\node[internalnode] (uu) at (-1,-12){};

\node[roots] (vv) at (-2,-15){};
\node[roots] (ww) at (-1,-15){};
\node[roots] (xx) at (0,-15){};

\foreach \alpha/\beta in {aa/bb,aa/cc,aa/dd,bb/ee,bb/ff,dd/gg,dd/hh,dd/ii,dd/jj,ee/kk,ee/ll,ee/mm,ee/nn,ee/oo,ff/pp,ff/qq,qq/rr,qq/ss,qq/tt,qq/uu,uu/vv,uu/ww,uu/xx}{
\draw[Plum,-] (\alpha)--(\beta);
}
\node at (0,-18){};

\end{tikzpicture}
\end{center}
\caption{Bijection between a dissection of an $18$-gon of type $\lambda=1+1+2+2+3+3+4=16$ and a rooted planar tree with downdegree sequence $d=(17,0,2,2,2,1,0,0,...,0)$.}
\label{fig:bijection2}
\end{figure}

Repeat this process for each of the edges in the dissection now connected to $v$. If a node is connected to an edge in the boundary of the $(n+2)$-gon, it will contribute one to the value of $r_0$ in the downdegree sequence (as it will be a leaf of the rooted tree, and has no further neighbors.) If a node is \emph{not} on the boundary of the $(n+2)$-gon, treat it as the new root node of a subtree, and repeat the first step.

For each $k_j\geq 1$ in $\lambda$, we will have $k_j=r_{j+1}$ vertices with downdegree $j$, and as each boundary edge of our $(n+2)$-gon (excepting our first root node) will have a leaf vertex placed on it, we must have $r_0=(n+2)-1=n+1$. The tree constructed has $k$ internal vertices, one for the root and $k-1$ for the diagonals, and $n+1$ leaves, so we have $n+k+1$ total vertices in our tree. Note that no vertices in this tree will have downdegree 1, so $r_1=0$.  So our constructed tree is in $T_{{\bf r},m}$ for ${\bf r}=(n+1,0,k_1,k_2,...,k_n)$ and $m=n+k$.

Note also that this bijection can easily be reversed: Given a rooted tree in $T_{{\bf r},n+k}$ with $k$ internal nodes each with some downdegree $j\geq 2$, we may place the internal node and its neighbors on the edges of a $(j+1)$-gon. Glue a pair of these polygons together along an edge if they share a vertex, and shift the resulting shape so that all edges that are unmatched form the boundary of a convex polygon. Of the internal nodes, only the original root node will appear on the boundary of this polygon. As there were $n+k+1$ original vertices and $k$ internal vertices, there must be $n+1$ vertices on the boundary apart from the root node. So we have constructed a polygonal dissection of an $(n+2)$-gon with $k_j$ $(j+2)$-gons, and our bijection is complete.
\end{proof}

Using results of Kreweras~\cite{kreweras} and Armstrong-Eu~\cite{armstrong}, the count for the number of rooted trees with a fixed downdegree sequence is known:

\begin{theorem}[Theorem 1.1~\cite{rhoades}]\label{thm:count1} Let $n\geq 1$, ${\bf v}=(1,2,...,n)$ and ${\bf r}=(r_1,...,r_n)$, such that ${\bf v}\cdot {\bf r}=n$. Set ${\bf r}!=r_1!r_2!\cdots r_n!$ and $|{\bf r}|=\sum r_j$. Then the number of rooted plane trees with $n+1$ vertices and downdegree sequence $(n-|{\bf r}|+1,r_1,r_2,...,r_n)$ is
$$A_{\bf r}({\bf v})=\frac{1}{1+n}\frac{(1+n)_{|{\bf r}|}}{{\bf r}!},$$
where $(y)_k=y(y-1)\cdots(y-k+1)=\frac{y!}{(y-k)!}$ is the falling factorial.
\end{theorem}

For a list of several other related classes of connected Catalan-type objects, enumerated by type and counted by the same formula, see the recent paper~\cite{rhoades} by Rhoades. From Theorem~\ref{thm:dissections2}, we have added a new class of objects (polygonal dissections of type $\lambda$) to their list. We will make use the of count provided by this bijection to prove our main theorem, showing the connection between the coefficients of reverse series of certain types and polygonal dissections.

\begin{restatable}[Polygonal Partitions]{theorem}{mainone}\label{thm:partitions}
The polynomial $x=z-\sum_{i=1}^r z^{d_i}$ with $2\leq d_1<d_2<\cdots<d_r$ has reverse series $z=\sum_{k=0}^{\infty} a_n x^{n+1}$ where $a_n$ counts the number of $(d_1,d_2,...,d_r)$-dissections of a convex $(n+2)$-gon.
\end{restatable}

Before delving into the connections between series reversions and polygonal dissections, we examine in Section~\ref{sec:mandelbrot} the initial problem that led us to consider reversions of series of the form $x=z-z^d$.

\section{Iterated Mandelbrot Polynomials and Reversions of Series}\label{sec:mandelbrot}

The work in this paper began initially as a study of the coefficients of iterated Mandelbrot (and $d$-multibrot) polynomials.

\begin{definition}[Mandelbrot and $d$-multibrot Polynomials]\label{def:mandelpoly} For variables $x,z\in\CC$, we define the Mandelbrot polynomial $f_x$ by $f_x(z)=z^2+x$. For $d\geq 2$, we define the $d$-multibrot polynomial $f_{d}$ to be the map $f_{d,x}(z)=z^d+x$.
\end{definition}

Of particular interest in complex dynamics is the orbit of $0$ under $f_{x}$ or $f_{d,x}$.  We were interested in a formula for the coefficients of the power series in $x$ of the infinitely iterated $d-$multibrot polynomial
$$f^{(\infty)}_d(x)=\lim_{n\rightarrow\infty} f^{(n)}_d(x),$$
where $f_d^{(n)}(x)$ is defined recursively by the formula $f_d^{(0)}(x)=0$ and $f_d^{(n)}(x)=\left(f_d^{(n-1)}(x)\right)^d+x$ for $n\geq 1$. Note that if the subscript $d$ is omitted, we assume that $d=2$.

The power series obtained by considering the limit of the iteration of zero under the $d$-multibrot polynomial $f^{(\infty)}_d(x)=\sum_{k}a_kx^{k+1}$ must satisfy
$$f^{(\infty)}_d(x)=\left(f^{(\infty)}_d(x)\right)^d+x.$$
Setting $z=f^{(\infty)}_d(x)$, we see that calculating the coefficients of $x$ in $f^{(\infty)}_d(x)$ is equivalent to computing the series reversion of the polynomial $x=z-z^d.$ With this in mind, here we introduce a version of the Lagrange Inversion formula to explicitly calculate the coefficients of $f_d^{(\infty)}(x)=\sum_{k=0}^{\infty}a_k x^{k+1}$.

\begin{theorem}[Lagrange Inversion Formula~\cite{Muller:1985ph}]\label{thm:inversion} Let $x$ be a (convergent) power series
$$x=z\left(1-\sum_{n=1}^{\infty}b_nz^n\right),$$
with reverse series
$$z=x\left(1+\sum_{n=1}^{\infty}a_nx^n\right).$$
Then the coefficients $a_n$ are given in terms of the $b_n$ by
$$a_n=\frac{1}{n+1}\sum_{\lambda} \binom{n+k}{k}\binom{k}{k_1,k_2,...,k_n}\prod_{j=1}^nb_j^{k_j},$$
where the sum is taken across all partitions $\lambda$ of $n$ into $k_j$ parts of size $j$ and $k$ total parts, e.g. across all nonnegative integer $n$-tuples $\{k_1,k_2,...,k_n\}$ such that
\begin{align*}
\sum_{j=1}^n k_j&=k\\
\sum_{j=1}^n k_j\cdot j&=n.
\end{align*}
\end{theorem}

As an immediate application of the Lagrange Inversion Formula, we produce the series reversion of the polynomial $x=z-z^d$:

\begin{theorem}\label{thm:main1} The polynomial $z=z^d+x$ has inverse series solution
$$z=\sum_{k=0}^{\infty} C_k^{(d)}x^{k(d-1)+1}.$$
\end{theorem}

\begin{proof}[Proof of Theorem~\ref{thm:main1}]
Fairly immediate from noting that only for $j=d-1$ are $b_j$ nonzero (specifically $b_{d-1}=1$). So all parts in partitions $\lambda$ contributing to the sum are of size $(d-1)$, and nonzero $a_n$ must be of form $n=k(d-1)$. From the Lagrange inversion formula in Theorem~\ref{thm:inversion}, these coefficients must then be:
\begin{align*}
a_n&=\frac{1}{n+1}\binom{n+k}{k}\\
&=\frac{1}{k(d-1)+1}\binom{k(d-1)+k}{k}\\
&=\frac{1}{k(d-1)+1}\binom{kd}{k}\\
&=C_k^{(d)}
\end{align*}
So the only nonzero terms in our series reversion are of the form
$$a_nx^{n+1}=a_{k(d-1)}x^{k(d-1)+1}=C_k^{(d)}x^{k(d-1)+1},$$
and we have our inverse series for $x=z-z^d$.
\end{proof}

As a corollary, we have that the coefficients of the infinitely iterated $d$-multibrot polynomials are given by the Fuss-Catalan numbers $C_k^{(d)}$. While this result was found and proved independently by the authors, the following statement appears to be fairly well-known for $d=2,3$ (see the OEIS at A001764. ) The formula holds in general for all $d\geq 2$.

\begin{corollary}[Coefficients of Infinitely Iterated $d$-multibrot Polynomials]\label{cor:maincoeff} Let $f_d^{(n)}(x)$ be defined recursively by the formula $f_d^{(0)}(x)=0$ and $f_d^{(n)}(x)=\left(f_d^{(n-1)}(x)\right)^d+x$ for $n\geq 1$, and set ${\displaystyle f_d^{(\infty)}(x)=\lim_{n\rightarrow\infty} \left(f^{(n)}_d(x)\right)}.$
Then
$$f_d^{(\infty)}(x)=\sum_{k=0}^{\infty} C_k^{(d)}x^{k(d-1)+1}.$$
\end{corollary}

\begin{proof}Immediate from Theorem~\ref{thm:main1} and the fact that $z=f^{(\infty)}_d(x)=\sum_{k}a_kx^{k+1}$ must satisfy
$$f^{(\infty)}_d(x)=\left(f^{(\infty)}_d(x)\right)^d+x,$$
or $z=z^d+x$. Note that this corollary could also be proved fairly directly via induction and the general recursion formula for the Fuss-Catalan series found in Theorem~\ref{thm:gencatrec}.
\end{proof}

As a further interesting note from this, the sum of the series formula for $z$ found in the Mandelbrot case gives a formula for the two fixed points of the (filled) Julia set ${\mathcal J}_x$ (the set of all points $z\in\CC$ such that the orbit of $0$ remains bounded under iterations by $f_x(z)=z^2+x$) for a fixed $x\in\CC$. See~\cite{milnor} for more details on dynamical systems and their fixed points.

\begin{remark}[Fixed Points of Filled Julia Sets ${\mathcal J}_x$] The series reversion of $z=z^2+x$ is
\begin{align*}
z&=\sum_{k=0}^{\infty} C_k x^{k+1}\\
&=x\sum_{k=0}^{\infty} C_k x^{k}\\
&=\frac{2x}{1+\sqrt{1-4x}}.
\end{align*}
For a fixed $x\in\CC$, the two complex values taken on by $\frac{2x}{1+\sqrt{1-4x}}$ each correspond to a separate fixed point of the Mandelbrot map, one each for the stable and unstable fixed point $z$ of $f_x(z)=z^2+x$ in ${\mathcal J}_x$.
\end{remark}

\section{Iterations of General Polynomials and Polygonal Dissections}\label{sec:main}

To return to polygonal partitions and their connections to the reversions of series, we note that Theorem~\ref{thm:main1} gives us immediately that $d$-dissections of polygons are counted by the coefficients of the series inverse of $x=z-z^d$.

\begin{corollary}\label{cor:main1} The coefficients $a_k$ of the series inversion $z=\sum_{k=0}^{\infty}a_kx^{k+1}$ of the polynomial $z=z^d+x$ count the number of $(d+1)$-gon polygonal partitions of a $(k+2)$-gon.
\end{corollary}

\begin{proof}[Proof of Corollary~\ref{cor:main1}] From Theorem~\ref{thm:main1} we have that the coefficients of the series reversion of $x=z-z^d$ are the Fuss-Catalan numbers $C_n^{(d)}=\frac{1}{(d-1)n+1}\binom{nd}{d}$. The corollary is immediate from Theorem~\ref{thm:bijection1}, as the Fuss-Catalan numbers enumerate $d$-partitions of $(n+2)$-gons.
\end{proof}

This will be a special case of our main theorem:

\mainone*

We begin our proof with a lemma counting the number of polygonal dissections of a fixed type (with a fixed number of each type of $(d+1)$-gon appearing in the dissection.)

\begin{lemma}\label{lem:fixpart} Fix integers $2\leq d_1<d_2<\cdots<d_r$, integer $n\geq 0$, and a partition $\lambda$ of $n$ with parts of sizes $j\in\{d_1-1,d_2-1,...,d_r-1\}$. Let $k_j$ for $1\leq j\leq r$ be the number of times that $j$ appears in $\lambda$, and let $k$ be the total parts in $\lambda$, i.e.
\begin{align*}
n&=\sum_{j=1}^{r}(d_j-1)k_{j}\text{ and}\\
k&=\sum_{j=1}^rk_j.
\end{align*}
Then the number of all polygonal dissections of type $\lambda$ of an $(n+2)$-gon is given by
$$a_{\lambda}=\frac{1}{n+1}\binom{n+k}{k}\binom{k}{k_1,k_2,...,k_r}.$$
\end{lemma}

\begin{proof}[Proof of Lemma~\ref{lem:fixpart}] From Theorem~\ref{thm:dissections2}, we know that the number of polygonal dissections of type $\lambda$ above is in bijection with the set of rooted planar trees with downdegree sequence ${\bf r}=(n+1,0,k_1,k_2,...,k_n)$ and $n+k+1$ vertices. From Theorem~\ref{thm:count1}, we know that the count of such planar trees is
\begin{align*}
A_{\bf r}({\bf v})&=\frac{1}{1+n+k}\frac{(1+n+k)_{k}}{0!k_1!k_2!\cdots k_n!},\\
&=\frac{1}{n+k+1}\frac{(n+k+1)!}{k_1!k_2!\cdots k_n!(n+1)!},\\
&=\frac{1}{n+1}\frac{(n+k)!}{k_1!k_2!\cdots k_n!n!},\\
&=\frac{1}{n+1}\frac{(n+k)!}{k!\;n!}\frac{k!}{k_1!k_2!\cdots k_n},\\
&=\frac{1}{n+1}\binom{n+k}{k}\binom{k}{k_1,k_2,...,k_n}.
\end{align*}
This proves the count $a_{\lambda}$ given in the statement of the theorem.
\end{proof}

With this in hand, we return to the proof of the main theorem.

\begin{proof}[Proof of Theorem~\ref{thm:partitions}] Given an $(n+2)$-gon, any fixed partition $\lambda$ of $n$ into positive integer parts of sizes chosen from the set $\{d_1-1,d_2-1,...,d_r-1\}$ corresponds to some fixed \emph{type} of polygonal $(d_1,d_2,...,d_r)$-dissection. From Lemma~\ref{lem:fixpart}, we know that there are
$$a_{\lambda}=\frac{1}{n+1}\binom{n+k}{k}\binom{k}{k_{d_1-1},k_{d_2-1},...,k_{d_r-1}}$$
such dissections, where $k_{d_j-1}$ parts of size $d_j-1$ appear in partition $\lambda$. Note that we have changed from $k_j$ to $k_{d_j-1}$ to better match the notation used in the statement the Lagrange inversion theorem.

Examining our polynomial
$$x=z-\sum_{i=1}^r z^{d_i}=z\left(1-\sum_{i=1}^r z^{d_i-1}\right),$$
we see that in the notation of the Lagrange inversion formula given in Theorem~\ref{thm:inversion}, only the only nonzero $b_j$ are those with $j=d_i-1$ for some $1\leq i\leq r$. So the coefficients $a_n$ of the reverse series $z=\sum_{i=0}^{\infty} a_n x^{n+1}$ are of the form
$$a_n=\frac{1}{n+1}\sum_{\lambda} \binom{n+k}{k}\binom{k}{k_1,k_2,...,k_n},$$
where the sum is taken across partitions $\lambda$ of the form
$$n=k_{d_1-1}(d_1-1)+k_{d_2-1}(d_2-1)+\cdots+ k_{d_r-1}(d_r-1).$$
Note that
\begin{align*}
a_n&=\frac{1}{n+1}\sum_{\lambda} \binom{n+k}{k}\binom{k}{k_1,k_2,...,k_n}\\
&=\sum_{\lambda} \frac{1}{n+1}\binom{n+k}{k}\binom{k}{k_{d_1-1},k_{d_2-1},...,k_{d_r-1}}\\
&=\sum_{\lambda} a_{\lambda},
\end{align*}
and our coefficients $a_n$ can be calculated by summing over all possible types of dissections in ${\lambda}$, made from parts of size $(d+1)$, with $d\in\{d_1,d_2,...,d_r\}$. This completes the proof.
\end{proof}

\begin{example} Consider the polynomial $f(z)=z^3+z^2+x$. The coefficients of the infinitely iterated polynomial are given by the series reversion of $z=z^3+z^2+x$, or $x=z-z^3-z^2$:
$$z=x+x^2+3x^3+10x^4+38x^5+154x^6+654x^7+\cdots$$
\vskip1ex
These coefficients $a_n$ count the number of dissections of an $(n+2)$-gon into triangles (3-gons) and quadrilaterals (4-gons).

As there are no $2$-gons, there is one way to cover the empty object with triangles or squares, so the coefficient of $x$ is $a_0=1$. For $n=1,2,3$, we see:

\begin{figure}[h!]
\begin{center}
\begin{tikzpicture}[scale=0.5,auto=left,vertices/.style={circle, fill=black, inner sep=0.5pt}]
\draw(0.75,0)--++(120:1.5)--++(240:1.5)--cycle;
\node at (0,-1.5) {$a_1=1$};
\end{tikzpicture}\;\;\;\;\;\;\;\;\begin{tikzpicture}[scale=0.5,auto=left,vertices/.style={circle, fill=black, inner sep=0.5pt}]
\draw(0.75,0)--++(90:1.5)--++(180:1.5)--++(270:1.5)--cycle;

\draw(0.75-1.9,0)--++(90:1.5)--++(180:1.5)--++(270:1.5)--cycle;
\draw[color=MidnightBlue] (0.75,0)--++(135:2.12);

\draw(0.75+1.9,0)--++(90:1.5)--++(180:1.5)--++(270:1.5)--cycle;
\draw[color=MidnightBlue] (0.75+0.4,0)--++(45:2.12);

\node at (0,-1.5) {$a_2=3$};
\end{tikzpicture}\;\;\;\;\;\;\;\;\begin{tikzpicture}[scale=0.5,auto=left,vertices/.style={circle, fill=black, inner sep=0.5pt}]

\draw (0,1.8)--++(72:1)--++(144:1)--++(216:1)--++(288:1)--cycle;
\draw[color=MidnightBlue](0,1.8)++(144:1.6)--++(0:1.6);
\draw[color=MidnightBlue](0,1.8)++(144:1.6)--++(324:1.6);

\draw (2.1,1.8)--++(72:1)--++(144:1)--++(216:1)--++(288:1)--cycle;
\draw[color=MidnightBlue] (2.1,1.8)++(108:1.6)--++(252:1.6);
\draw[color=MidnightBlue] (2.1,1.8)++(108:1.6)--++(288:1.6);

\draw (4.2,1.8)--++(72:1)--++(144:1)--++(216:1)--++(288:1)--cycle;
\draw[color=MidnightBlue] (4.2,1.8)++(72:1)--++(180:1.6);
\draw[color=MidnightBlue] (4.2,1.8)++(72:1)--++(216:1.6);

\draw (6.3,1.8)--++(72:1)--++(144:1)--++(216:1)--++(288:1)--cycle;
\draw[color=MidnightBlue] (6.3,1.8)--++(144:1.6);
\draw[color=MidnightBlue] (6.3,1.8)--++(108:1.6);

\draw (8.4,1.8)--++(72:1)--++(144:1)--++(216:1)--++(288:1)--cycle;
\draw[color=MidnightBlue] (8.4,1.8)++(0:-1)--++(72:1.6);
\draw[color=MidnightBlue] (8.4,1.8)++(0:-1)--++(36:1.6);

\draw (0,0)--++(72:1)--++(144:1)--++(216:1)--++(288:1)--cycle;
\draw[color=MidnightBlue](0,0)++(144:1.6)--++(0:1.6);
%\draw[color=MidnightBlue](0,0)++(144:1.6)--++(324:1.6);

\draw (2.1,0)--++(72:1)--++(144:1)--++(216:1)--++(288:1)--cycle;
%\draw[color=MidnightBlue] (2.1,0)++(108:1.6)--++(252:1.6);
\draw[color=MidnightBlue] (2.1,0)++(108:1.6)--++(288:1.6);

\draw (4.2,0)--++(72:1)--++(144:1)--++(216:1)--++(288:1)--cycle;
%\draw[color=MidnightBlue] (4.2,0)++(72:1)--++(180:1.6);
\draw[color=MidnightBlue] (4.2,0)++(72:1)--++(216:1.6);

\draw (6.3,0)--++(72:1)--++(144:1)--++(216:1)--++(288:1)--cycle;
\draw[color=MidnightBlue] (6.3,0)--++(144:1.6);
%\draw[color=MidnightBlue] (6.3,0)--++(108:1.6);

\draw (8.4,0)--++(72:1)--++(144:1)--++(216:1)--++(288:1)--cycle;
\draw[color=MidnightBlue] (8.4,0)++(0:-1)--++(72:1.6);
%\draw[color=MidnightBlue] (8.4,0)++(0:-1)--++(36:1.6);

\node at (4,-0.75) {$a_3=10$};

\end{tikzpicture}
\end{center}
\caption{$(2,3)$-dissections of $n$-gons for $n=1,2,3$.}
\label{fig:dissection count}
\end{figure}
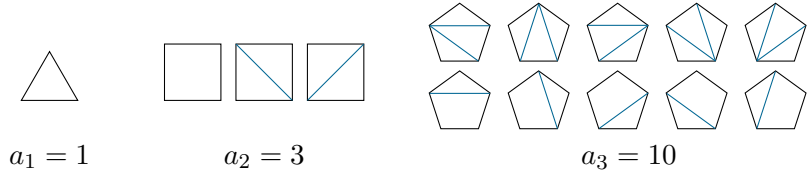
\end{example}

Extending this slightly, we have a power series whose reverse series has coefficients counting all dissections of an $(n+2)$-gon by noncrossing diagonals.

\begin{theorem}[Super-Catalan Numbers and Series Reversions]\label{thm:generaldissections}The power series $x=z-\sum_{j=1}^{\infty} z^{j}$ has reverse series $z=\sum_{k=0}^{\infty} a_n x^{n+1}$ where $s_n$ counts the all possible subsets of noncrossing diagonals a convex $(n+2)$-gon. The coefficient $s_n$ is given by
$$s_n=\frac{1}{n+1}\sum_{\lambda} \binom{n+k}{k}\binom{k}{k_1,k_2,....,k_n},$$
where the sum is taken across all partitions $\lambda$ of $n$ with $k_j$ parts of size $j$ and $k$ total parts.
\end{theorem}

\begin{proof} From the Lagrange inversion theorem, we know that the coefficient $s_n$ in the reverse power series of $x=z-\sum_{j=1}^{\infty}z^j$ must be of the form given in the statement of the theorem. All partitions $\lambda$ of $n$ contributing to the sum must have parts at \emph{most} $n$, so the $s_n$ above must be the same as the coefficient $a_n$ for the reverse series of the polynomial $x=z-\sum_{j=1}^{n+1} z^j$. From Theorem~\ref{thm:partitions}, we know that the coefficients $a_n$ of the reversion of the polynomial with nonzero terms $z^2,z^3,\cdots,z^{n+1}$ enumerate $(2,3,....,n+1)$-dissections of an $(n+2)$-gon -- a set which includes all possible polygonal dissections.
\end{proof}

The set of all polygonal dissections of an $(n+2)$-gon is counted by the super-Catalan numbers $s_n$ (also called the Schr\"{o}der-Hipparchus numbers.) (See~\cite{supercatalan} for an extensive list of other families of sets counted by $s_n$.) While several other formulas for the super-Catalan are known, Theorem~\ref{thm:generaldissections} gives a nice decomposition of $s_n$, summed across structures indexed by partitions $\lambda$ of $n$.

\section{Generalizations to Colored Dissections}\label{sec:main2}

The coefficients of slightly more general series reversions can be immediately interpreted using the formula in Theorem~\ref{thm:partitions}.

\begin{definition}A \emph{colored polygonal dissection} is a polygonal dissection where each $(d+1)$-gon appearing in the dissection can be assigned $b_{d}$ possible colors for $d\geq 2$.
\end{definition}

\begin{theorem}[Colored Polygonal Partitions]\label{thm:coloredpartitions}
The polynomial $x=z-\sum_{i=1}^r b_{d_i}z^{d_i}$ with $d_1>d_2>\cdots>d_r\geq 2$ and $b_{d_i}\geq 1$ for all $1\leq i\leq r$ has reverse series $z=\sum_{k=0}^{\infty} a_n x^{n+1}$ where $a_n$ counts the number of colored polygonal $(d_1,d_2,...,d_r)$-dissections of a convex $(n+2)$-gon.
\end{theorem}

\begin{proof} From Lemma~\ref{lem:fixpart}, we know that the number of $(d_1,d_2,...,d_r)$-partitions of an $n$-gon with precisely $k_j$ of the $(d_j+1)$-gons appearing in the dissection for $1\leq j\leq r$ is given by
$$a_{\lambda}=\frac{1}{n+1}\binom{n+k}{k}\binom{k}{k_{d_1-1},k_{d_2-1},...,k_{d_r-1}}.$$
If each $(d_i+1)$-gon can be assigned one of $b_{d_i}$ colors, then there are
$$a_{\lambda}^{\ast}=\frac{1}{n+1}\binom{n+k}{k}\binom{k}{k_{d_1-1},k_{d_2-1},...,k_{d_r-1}}\prod_{i=1}^r b_{d_i}^{k_{d_i-1}}$$
such colored dissections, as we have $b_{d_i}$ choices for each of $k_{d_i-1}$ of the $(d_i+1)$-gons appearing in a given dissection.

As in the proof of Theorem~\ref{thm:partitions}, we have that the coefficients of the inverse series of the polynomial $x=z-\sum_{i=1}^r b_{d_i}z^{d_i}$ must be:
\begin{align*}
a_n&=\frac{1}{n+1}\sum_{\lambda} \binom{n+k}{k}\binom{k}{k_1,k_2,...,k_n}\prod_{j=1}^n b_{j+1}^{k_j}\\
&=\sum_{\lambda} \frac{1}{n+1}\binom{n+k}{k}\binom{k}{k_{d_1-1},k_{d_2-1},...,k_{d_r-1}}\prod_{i=1}^r b_{d_r}^{k_{d_r-1}}\\
&=\sum_{\lambda} a_{\lambda}^{\ast},
\end{align*}
completing our proof.
\end{proof}

\section{Further Questions}

This paper provides a complete combinatorial interpretation of series reversions of polynomials of the form
$$z=b_1z^{d_1}+b_2z^{d_2}+\cdots+b_rz^{d_r}+x$$
for positive integers $b_j$. As future work, we would be curious to see combinatorial approaches to the following question:

\begin{question} In general, given a pair of polynomials $g(z)$ and $h(x)$ with integer coefficients, is there a family of sets of objects ${\mathcal A}_{g,h}$ counted by the coefficients of the reversion of the power series $z=g(z)+h(x)$?
\end{question}

This question is answered here for $h(x)=x$ and $g(z)$ with \emph{positive} integer coefficients and all terms of degree at least two, but remains open in other cases. Series other than the generating functions of Catalan-type objects may appear as series inversions using similar iterative techniques, and we would be interested in seeing other classes of objects enumerated by such coefficients.

\bibliography{FussCatalan}
\bibliographystyle{plain}

\end{document}